%% file: chem_agg1.tex
\documentclass{ejm}

\input{ejmcommands}

\usepackage{graphicx}
\usepackage{epsfig}

\newtheorem{proposition}{Proposition}[section]
\newtheorem{lemma}{Lemma}

\newdefinition{definition}[theorem]{Definition}
\newdefinition{assumption}{Assumption}[section]
\newdefinition{assumptions}{Assumptions}[section]
\newdefinition{remark}{Remark}

\def\Rset{\mathbb{R}}
\def\ds{\displaystyle}

\title[Properties of the chemostat model with aggregated biomass]{Properties of the chemostat model
  with aggregated biomass}

\author[A. Rapaport]{Alain Rapaport$\,^1$}

\affiliation{$^1\,$ MISTEA, U. Montpellier, INRA, Montpellier SupAgro, France.
    email\textup{\nocorr: \texttt{alain.rapaport@inra.fr}}
}

\date{2 March 2018}
\pubyear{2018}
\volume{000}
\pagerange{\pageref{firstpage}--\pageref{lastpage}}

\begin{document}

\label{firstpage}
\maketitle

\begin{abstract}
We revisit the well-known chemostat model, considering that bacteria
can be attached together in aggregates or flocs. 
We distinguish explicitly free and attached
compartments in the model and give sufficient conditions for coexistence of
these two forms. We then study the case of fast attachment and
detachment and shows how it is related to density-dependent growth
functions.
Finally, we give some insights concerning the cases of multi-specific flocs and
different removal rates.
\end{abstract}


\section{Introduction}

Attachment and detachment phenomena of bacteria, whether in biofilms
on a support \cite{C95,IWA06} or in the form of aggregates or flocs
\cite{TJF99} are well known and frequently observed in bacterial growth. 
Nevertheless, it is only relatively recently that they have been
explicitly taken into account in chemostat-based mathematical models. 
The Freter model \cite{FBFVC83,JKLS03}, proposed in the 1980s as a functional model
of the intestine bacterial ecosystem, is one of the very first to
explicitly distinguish planktonic biomass from attached biomass. This
model considers specific attachment and detachment terms and has been
mathematically studied in a spatialized form by introducing advection
and diffusion terms \cite{BJS08}.
Several works in the biomathematical literature consider extensions to
the chemostat model spatialized with (fixed) attachment on a wall by
\cite{BS99,JKLS03,SS00}. In general, flocculation models describe the
dynamics of the distribution of flocs sizes \cite{TJF99} and their
influence on growth dynamics \cite{HLH07}, but comparatively there are
relatively few studies of simplified models that only distinguish two biomass compartments: planktonic and attached. In \cite{HR08}, it is shown for
such models that total biomass growth follows a density-dependent
distribution, under the assumption that attachment and detachment
velocities are large compared to biological terms. This is in
accordance with experimental observations that have showed that the
kinetics of processes with attached biomass are better represented by
ratio-dependent \cite{HG07} expressions.

The purpose of the present work is to generalize the existing
results concerning these simplified models.

The majority of models of the literature consider explicit
attachment and detachment term expressions. We adopt here a more
general presentation
which does not particularize the specific attachment and detachment
kinetics terms 
and thus namely includes existing models \cite{TSJ97,PW99,JKLS03}. 
In every case, the assumptions
about faster growth and higher planktonic bacteria removal rate are justified by
experimental observations \cite{HMC09}. This allows us to consider reduced
models considering the total biomass instead of planktonic and
attached ones, which provides extensions of the well-know chemostat
model with unusual characteristics.

It should be observed that attachment and detachment
velocities can be of a very variable order of magnitude, according to procedures
and operating conditions \cite{BK95}, justifying the fact of considering
reduced models or not. 

\section{A general formulation}

Under certain growth conditions and in some environments, microbial species may
present aggregates of microorganisms or flocs of various sizes (see
Figure \ref{figattachement}). Microorganisms can also attach themselves to the walls of
tanks, pipes, reactors, etc. (or more generally of any chemostat-based
device), and thus create biofilms with varied thicknesses. Over time,
micro-organisms, parts of flocs or of biofilms, detach and are
  released in the liquid medium as isolated individuals or small-sized aggregates (see Figure \ref{figdetachement}).
These bacterial assemblages (which can be observed under the
microscope) affect the performance of chemostats at the macroscopic
level, namely regarding:
\begin{itemize}
\item the growth of biomass: bacterial individuals have differentiated
  access to biotic resource (substrate)
  depending on their position inside or on the periphery of assemblies. In addition,
  microorganism secretions of polymers that enable the attachment are generally achieved
  to the detriment of their growth.
\item the disappearance of biomass: flocs and biofilms are most often less likely to be
  dragged away by the chemostat outflow, comparatively to isolated
  individuals.
\end{itemize}
The appearance and evolution mechanisms of these assemblies, which at the same
time relate to biology, mechanics and hydrodynamics, are complex, partially understood
and difficult to be modeled at a microscopic scale. Our objective is
to study how the conventional model of the chemostat can be enriched with considerations
reflecting the effects of biomass attachment and detachment at the macroscopic
level (in other words, without representing all the refinements that a description would
bring at the microscopic level).
\begin{figure}[h!]
  \begin{center}
    \includegraphics[width=.4\textwidth]{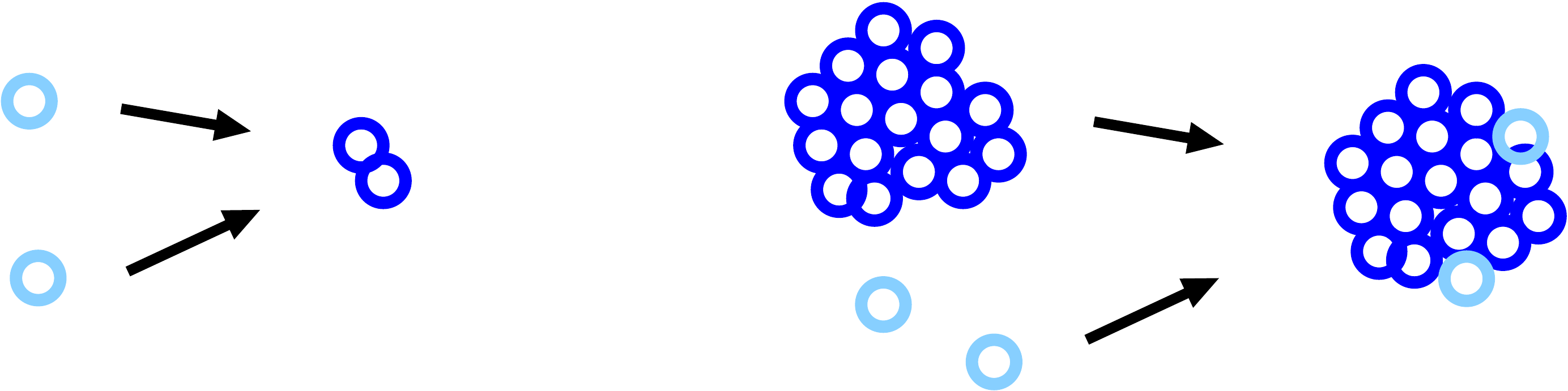}
    \caption{Isolated individuals may aggregate to form
      a floc, or else attach to an already formed aggregate.  \label{figattachement}}
  \end{center}
\end{figure}
\begin{figure}[h!]
\begin{center}
    \includegraphics[width=.4\textwidth]{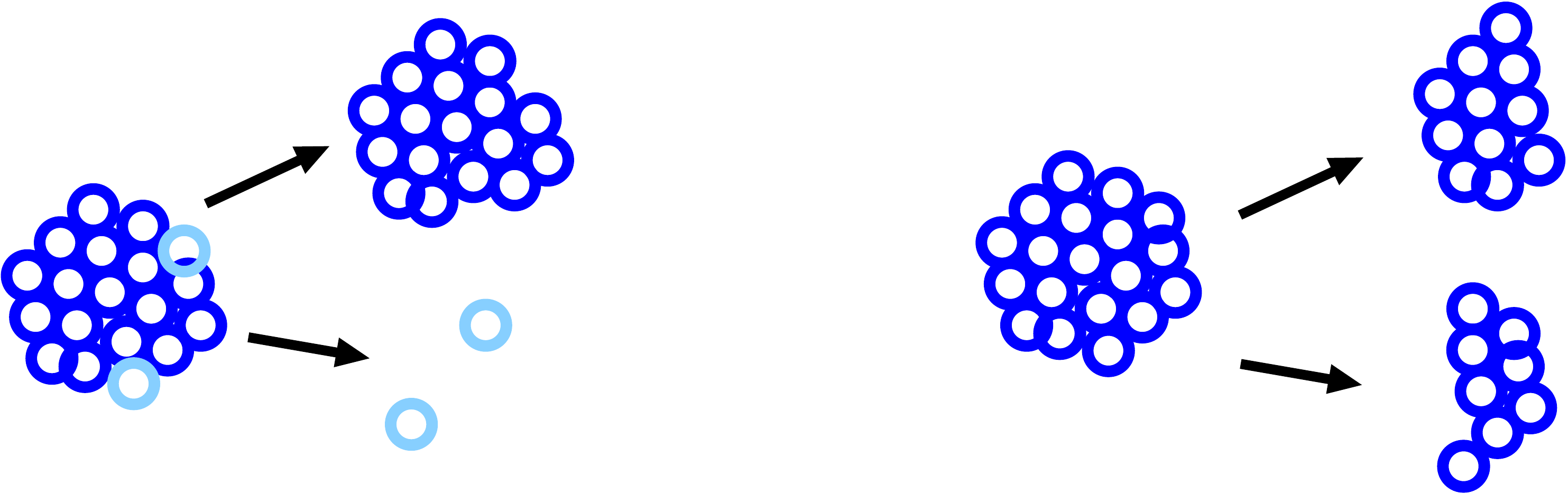}
    \caption{Individuals can detach from an aggregate.
      An aggregate can be split into smaller aggregates.
      \label{figdetachement}}
  \end{center}
\end{figure}

We consider that the total biomass of a given species is decomposed
into "planktonic" (or "free") biomass made up of non-attached
microorganisms (or at least that behave as such; which may still be
the case of small assemblies) and "aggregate" biomass (without accurately taking account of the shape and of the size of assemblies).
Thus, we write the concentration $x$ of the total biomass as the sum
of concentrations $u$ and $v$ of planktonic and aggregate biomass,
respectively:
\begin{equation}
  \label{somme}
  x=u+v \ .
\end{equation}
This distinction allows us to take into account different growth and death characteristics
according to whether microorganisms are attached or not. We thus denote
respectively by $\mu_{u}(\cdot)$, $D_{u}$ and $\mu_{v}(\cdot)$, $D_{v}$
the specific growth and removal rates of
planktonic and aggregate compartments. 
$D_{u}$ and $D_{v}$ are positive numbers and $\mu_{u}(\cdot)$,
  $\mu_{v}(\cdot)$ are smooth functions that verify
$\mu_{u}(0)=\mu_{v}(0)=0$ and positive away from zero.
On the other hand, we denote the specific
velocities of attachment of planktonic biomass by $\alpha(\cdot)$ and
by $\beta(\cdot)$ the ones of detachment
of the attached biomass. As a result, we obtain the following
chemostat model, where $s$ denotes the substrate concentration:
\begin{equation}
  \label{chem_attach}
  \left\{
  \begin{array}{lll}
    \ds \frac{ds}{dt} = & \ds D(S_{in}-s)-\mu_{u}(s)u-\mu_{v}(s)v\\[4mm]
    \ds \frac{du}{dt} =  & \mu_{u}(s)u-D_{u}u
    -\alpha(u,v)u+\beta(v)v\\[4mm]
    \ds \frac{dv}{dt} =  & \mu_{v}(s)v-D_{v}v
    +\alpha(u,v)u-\beta(v)v
  \end{array}\right.
\end{equation}
The positive parameters $D$ and $S_{in}$ denote the dilution rate and
input concentration of the substrate.
As usual in chemostat models, we take unit yield coefficients without loss of generality. The simplicity
of this representation, which does not account for the richness of forms and
possible sizes of aggregates, should be regarded as the considering of an average microorganism
behavior within aggregates or biofilms, which differs from that of isolated
microorganisms. Since it is difficult to obtain or to justify precise expressions of the
attachment and detachment terms for this type of model, our purpose is to understand
and qualitatively predict the possible effects of these terms on the dynamics of the
system (to this end, we will merely consider simple expressions as possible representatives).
It should be noted that the attachment and detachment terms depend on the
operating conditions (in particular the flow rate), that we
consider here to be fixed.

\medskip

We first show that the solutions of system (\ref{chem_attach}) stay non-negative
and bounded, as in the classical chemostat model. 

\begin{lemma}
The non-negative orthant $\Rset_{+}^3$ is forwardly invariant by the dynamics (\ref{chem_attach})
and any solution in this domain is bounded.
\end{lemma}

\begin{proof}
At $s=0$, one has $\dot s=DS_{in}>0$. Therefore $s$ stays positive.
One has $\frac{d}{dt}(u+v) \geq (\mu_{v}(s)-D_{u})(u+v)$, which
shows that $x=u+v$ stay positive. At $u=0$, resp. $v=0$, one
has $\frac{d}{dt}u \geq \beta(\cdot)x\geq 0$, resp. $\frac{d}{dt}v \geq
\alpha(\cdot)x \geq 0$. Therefore the variables $u$ and $v$ stay
non-negative. Finally, on has $\frac{d}{dt}(s+u+v)\leq
DS_{in}-D_{v}(s+u+v)$ which shows that the quantity $s+u+v$ is
bounded, and a consequence, $s$, $u$ and $v$ also.
\end{proof}

Hereafter, we consider the following assumptions, which reflect the
considerations discussed in the introduction:
\begin{assumptions}
  The kinetics functions $\mu_{u}(\cdot)$,  $\mu_{v}(\cdot)$,
  $\alpha(\cdot)$, $\beta(\cdot)$ and parameters $D$, $D_{u}$, $D_{v}$ fulfill the following properties.
  \label{hypagreg}
  \begin{itemize}
  \item[i.] The specific growth kinetics $\mu_{u}(\cdot)$ and
    $\mu_{v}(\cdot)$ are smooth increasing functions, null at zero,
    that verify:
    \begin{equation}
      \label{hypo_mu}
      \mu_{u}(s) > \mu_{v}(s), \quad \forall s>0
    \end{equation}
  \item[ii.] The removal rates of aggregate and planktonic biomass
    verify:
    \begin{equation}
      \label{hypoD}
      D \geq D_{u} \geq D_{v}>0
    \end{equation}
  \item[iii.] The function $\alpha$ only depends on concentrations $u$ and $v$ in an increasing manner
    and such that:
    \[
    u>0 \; \Rightarrow \alpha(u,0)>0
    \]
    with
    \[
    \frac{\partial\alpha}{\partial u}(u,v) \geq 
    \frac{\partial\alpha}{\partial v}(u,v), \quad \forall (u,v) .
    \]
  \item[iv.] The function $\beta$ depends only on the concentration $v$ in a decreasing manner
    and such that $v\mapsto \beta(v)v$ is increasing with:
    \[
    v>0 \; \Rightarrow \; \beta(v)>0 .
    \]
  \end{itemize}
\end{assumptions}

Typical instances of functions $\mu_{u}$, $\mu_{v}$ are given by
  the Monod expression
\[
\mu_{\max}\,\frac{s}{K_{s}+s}
\]
(with distinct values of the parameters
  $\mu_{\max}$, $K_{s}$ for planktonic and attached bacteria), that is
  quite popular in microbiology. 
  Assumption i. expresses the observation that attached bacteria have
  generally a more difficult
  acces to substrate. 
  With Assumption ii, we first neglect the mortality of planktonic bacteria,
  compared to the removal rate $D$, and considered that the
  substrate is the reactant that is removed most easily because of the
  the size of its molecules (that is usually much smaller that
  micro-organisms, justifying the assumption $D_{u}\leq D$). In a similar way, the
  attachment slows down the effective removal rate of the attached
  bacteria compared to the planktonic ones (which is represented by
  the inequality $D_{v}\leq D_{u}$).
  Typically, it can be considered that the specific attachment velocity $\alpha(u,v)$ can be
decomposed into a sum of two terms $\alpha_{u}(u)$ and $\alpha_{v}(v)$ that reflect the two possible
types of attachments: on free bacteria or on bacteria already in flocs. Considering
that free bacteria mainly attach on the surface of flocs, and that when the size of flocs
increases, the ratio surface over volume does not increase as quickly as the volume,
it can be expected that the function $\alpha_{v}$ increases more slowly than $\alpha_{u}$, which is then
reflected by $\alpha_{u}'(u)\geq \alpha_{v}'(v)$ for all $(u,v)$, justifying Assumption iii.
In general, it is expected that the detachment velocity $v \mapsto \beta(v)v$ increases with
the density $v$ of the attached biomass, but when the flocs size increases, the ratio surface
over volume increases more slowly than the volume, which results in a decrease of the
function $v \mapsto \beta(v)v$, thus justifying Assumption iv.

\section{Study of the coexistence between the two forms}

We assume that
\[
D=D_{u}=D_{v} ,
\]
(the more general case of different removal rates is discussed
in Section \ref{SectionDiffD}), which allows to consider the variable $z(t) = s(t) + x(t)$, a solution
of the differential equation :
\[
\frac{dz}{dt}=D(S_{in}-z) .
\]
whose solutions converge exponentially to $S_{in}$.
Therefore, the system (\ref{chem_attach}) has a cascade structure  in
the $(z,u,v)$ coordinates :
\begin{equation}
\label{3dsys}
\begin{array}{l}
\ds \frac{dz}{dt} = f_{0}(z)\\[3mm]
\ds \frac{du}{dt} = f_{1}(z,u,v) , \; \frac{dv}{dt} = f_{2}(z,u,v)
\end{array}
\end{equation}
and the local stability analysis of its equilibriums is given by the local
stability of the equilibriums of the reduced dynamics :
\begin{equation}
\label{2dsys}
\frac{du}{dt} = f_{1}(S_{in},u,v) , \;
\ds \frac{dv}{dt} = f_{2}(S_{in},u,v)
\end{equation}
The global behavior of the solutions of the system (\ref{3dsys}) is more
delicate to be deduced from the global behavior of the reduced system
(\ref{2dsys}) and relies on the theory of asymptotically autonomous
systems \cite{MST95}. However, we recall the well-known result when the reduced
system (\ref{2dsys})has a unique globally asymptotically stable
equilibrium, that states that any bounded solution of (\ref{3dsys})
converge to the unique equilibrium of (\ref{3dsys}).
We consider in the following the reduced dynamics of
(\ref{chem_attach}) for $z=S_{in}$:
\begin{equation}
  \label{chem_attach_reduit}
  \left\{
  \begin{array}{lll}
    \ds \frac{du}{dt} =  & \mu_{u}(S_{in}-u-v)u-Du
    -\alpha(u,v)u+\beta(v)v\\[4mm]
    \ds \frac{dv}{dt} =  & \mu_{v}(S_{in}-u-v)v-Dv
    +\alpha(u,v)u-\beta(v)v
  \end{array}\right.
\end{equation} 
We study the possible positive steady-states $(u^\star,v^\star)$ of this system, that is to say, the positive
solutions of the system:
\begin{equation}
  \label{sys-equ}
  \left\{
  \begin{array}{l}
    \ds \mu_{u}(S_{in}-u-v)u-Du -\alpha(u,v)u+\beta(v)v=0\\[4mm]
    \ds \mu_{v}(S_{in}-u-v)v-Dv +\alpha(u,v)u-\beta(v)v=0
  \end{array}\right.
\end{equation}
It can be immediately noticed that $u^\star=0$ implies
$\beta(v^\star)v^\star=0$ and $v^\star=0$, 
$\alpha(u^\star,0)u^\star=0$. The assumptions \ref{hypagreg} that we consider on terms $\alpha(\cdot)$ and
$\beta(\cdot)$
then allow us to infer that there is no steady-state where only one of the two forms
would be present.

\subsection{Coexistence steady-state}
\label{sec_coex_agreg}
Adding equations (\ref{sys-equ}), we obtain $(u^\star,v^\star)$
as a solution of the system:
\[
\left\{
\begin{array}{rcrll}
  (\mu_{u}(s)-D)u & + & (\mu_{v}(s)-D)v & = & 0\\
  u & + & v & = & S_{in}-s
\end{array}\right.
\]

Consequently, a coexistence steady-state (if it exists) verifies:
\begin{equation}
  \label{xpxa}
  u^\star=(S_{in}-s^\star)\frac{D-\mu_{v}(s^\star)}{\mu_{u}(s^\star)-\mu_{v}(s^\star)},
  \quad
  v^\star=(S_{in}-s^\star)\frac{\mu_{u}(s^\star)-D}{\mu_{u}(s^\star)-\mu_{v}(s^\star)}
\end{equation}
with $s^\star=S_{in}-u^\star-v^\star$.
According to hypothesis (\ref{hypo_mu}), we obtain the following
necessary condition:
\[
\mu_{u}(s^\star)>D>\mu_{v}(s^\star) .
\]
By defining the break-even concentration by $\lambda_{u}$,
$\lambda_{v}$ for the dilution rate $D$ (that is that verify 
$\mu_{u}(\lambda_{u})=\mu_{v}(\lambda_{v})=D$ with $\lambda_{v}>
\lambda_{u}$, see \cite{SW95,HLRS17}), we deduce
that a coexistence steady-state must verify:
\[
s^\star \in (\lambda_{u},\lambda_{v}) .
\] 
Thus, a necessary condition for the existence of a coexistence
steady-state is:
\begin{equation}
  \label{condlambda}
  \lambda_{u}<S_{in} .
\end{equation}
At this stage, it is difficult to prove the existence of solutions without specifying
attachment and detachment functions $\alpha(\cdot)$
and $\beta(\cdot)$. If we consider that we are only
dealing with flocs of small size, as a first approximation it is possible to assume that
$\alpha$ is a function of $x = u+v$ (that is, functions $\alpha_{u}$
and $\alpha_{v}$ are identical), which will be
chosen as linear (to simplify), and that the function $\beta$ does not
depend of $v$:
\begin{equation}
  \label{alpha_beta_simples}
  \alpha(u,v)=a(u+v)=ax, \quad \beta(v)=b
\end{equation}
where $a$ and $b$ are two positive constants. Thereby, the hypotheses
\ref{hypagreg} are correctly verified.

\begin{proposition}
  \label{propexist}
  For growth functions $\mu_{u}$, $\mu_{v}$ that verify point i) of Assumptions
  \ref{hypagreg} and attachment and detachment functions $\alpha(\cdot)$, $\beta(\cdot)$ of the form (\ref{alpha_beta_simples}), there
  exists a unique coexistence steady-state of system
  (\ref{chem_attach}) if and only if the condition :
  \begin{equation}
    \label{condexist}
    D <\mu_{u}(S_{in})
  \end{equation}
  is verified.
\end{proposition}

\begin{proof}
As mentioned previously, it is enough to show the existence of
  a positive equilibrium of the reduced dynamics (\ref{chem_attach_reduit}).
  $I$ denotes the interval :
  $$I= ]\lambda_u,\lambda_v[.$$
  To simplify the writing, the following notations are introduced:
  $$
  \varphi_u(s)=\mu_u(s)-D \quad \mbox{and} \quad \varphi_v(s)=\mu_v(s)-D.
  $$
  For all  $s\in I$, we have  $\varphi_u(s)>0>\varphi_v(s)$. The steady-states $(s^*,u^*,v^*)$ are
  given by:
  \begin{eqnarray}                       \label{IsoFlocGen}
    \left\{
    \begin{array}{lll}
      0=\varphi_u(s^*)u^*-a(u^*+v^*)u^*+b v^*\\[1mm]
      0=\varphi_v(s^*)v^*+a(u^*+v^*)u^*-b v^*.
    \end{array}
  \right.
\end{eqnarray}
If $u^*=0$ then, from the first equation, it can be deduced
that $v^*=0$. Similarly, if $v^*=0$ then, from the second equation it can be deduced that
$u^*=0$. Consequently, the steady-states are the washout $E_0=(S_{in},0,0)$ or a
steady-state of the form:
$$E^*=(s^*,u^*,v^*)$$
with $u^*>0$ , $v^*>0$ and $s^*=S_{in}-u^*-v^*$.
In order to solve Equations (\ref{IsoFlocGen}), one uses a method
similar to the characteristic at steady-state method. This
method consists in determining the steady-states of the system formed by the 2nd
and 3rd equations of (\ref{chem_attach}), where the variable $s$ is considered to be an input of the
system. In other words the aim is to solve the system formed by the first and the
second equation of (\ref{IsoFlocGen}), in which $u^*$ and $v^*$ are the unknowns and $s^*$ is considered
as being a parameter. It thus yields :
$$u^*=U(s^*),\qquad v^*=V(s^*).$$ 
If $u^*$ and $v^*$ are replaced by these expressions in the first equation of (\ref{chem_attach}), an
equation of the single variable $s^*$ is obtained of the form:
$$D(S_{in}-s^*)=H(s^*)\quad\mbox{with}\quad
H(s^*)=\mu_u(s^*)U(s^*)+\mu_v(s^*)V(s^*) $$
that is solved, see Figure \ref{figsol}, to find a positive solution $s^*$ . This solution gives a
positive steady-state, provided that $U(s^*)$ and $V(s^*)$ be positive. In the following, the
functions U, V and H are determined and the conditions are given in order for the
solution $s^*$ to exist.
\begin{figure}[ht]
  \setlength{\unitlength}{1.0cm}
  \begin{center}
    \begin{picture}(6.7,6)(0,0)
      \put(0,6.5){\rotatebox{-90}{\includegraphics[width=6cm,height=8cm]{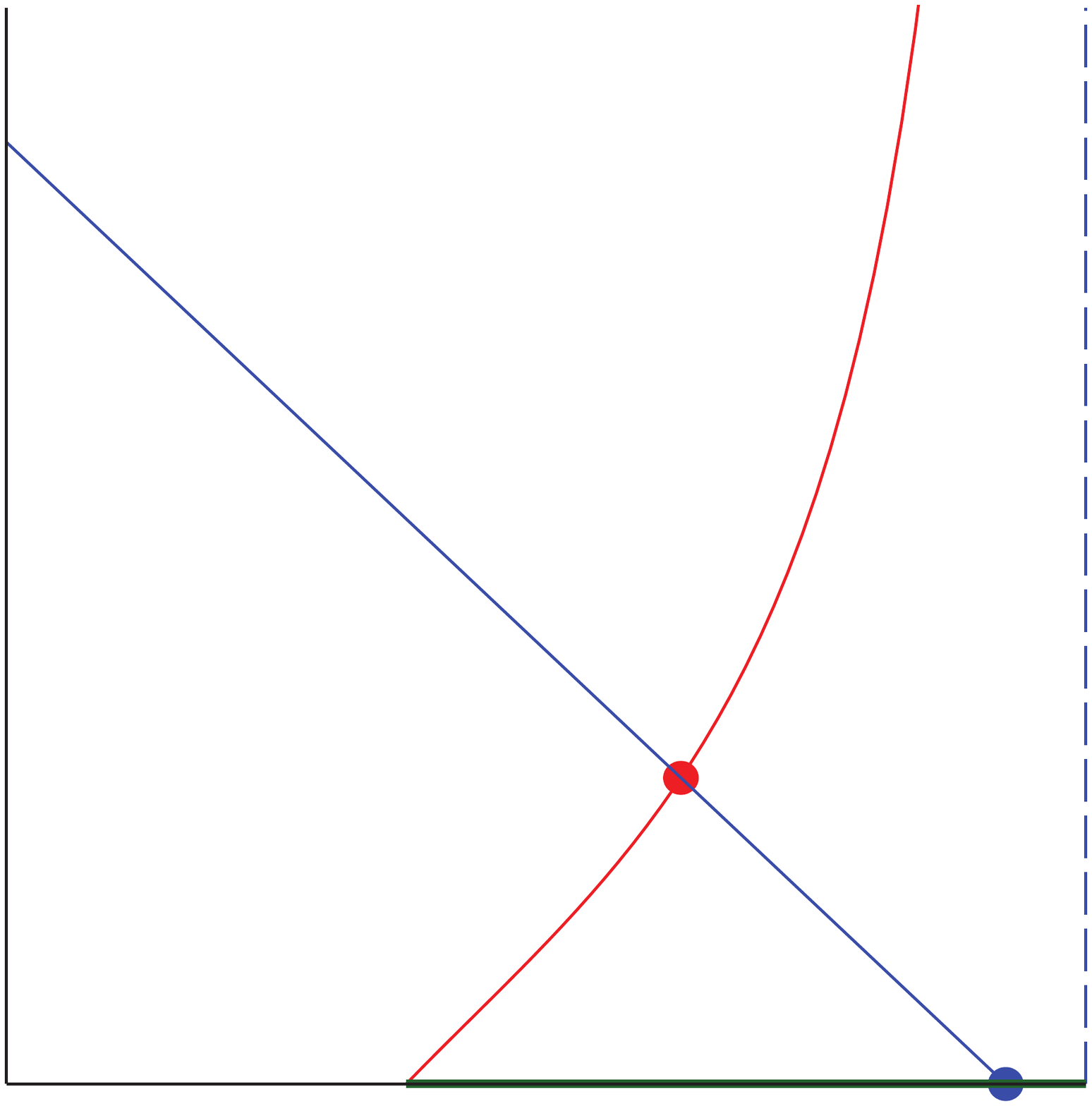}}}
      \put(5.5,4.7){{ ${H(s)}$}}
      \put(2.9,3.9){{ ${D(S_{in}-s)}$}}
      \put(0.5,5.2){{ $DS_{in}$}}
      \put(4.7,2.2){{ ${E^*}$}}
      \put(5.9,0.9){{ ${E_0}$}}
      \put(5.7,0.4){{ $S_{in}$}}
      \put(3.1,0.4){{ ${\lambda_u}$}}
      \put(6.3,0.4){{ ${\lambda_v}$}}
      \put(6.5,0.7){{ $s$}}
    \end{picture}
  \end{center}
  \caption{Existence of a unique positive steady-state.} \label{figsol}
\end{figure}

By summing the 1st and 2nd equations (\ref{IsoFlocGen}), we obtain:
\begin{equation}                         \label{eqPhiPsi}
  \varphi_u(s^*)u^* + \varphi_v(s^*)v^*=0.
\end{equation}

This equation admits a positive solution if and only if $\varphi_u(s^*)$ and 
$\varphi_v(s^*)$ are of
opposite signs, that is, if and only if $s^*\in I$. If this equation admits a solution in this
interval then Equation (\ref{eqPhiPsi}) can be written as follows :
\begin{equation}                           \label{eqV-U}
  v^*=-\frac{\varphi_u(s^*)}{\varphi_v(s^*)}u^*.
\end{equation}
By replacing $v^*$ by Expression (\ref{eqV-U}) in the first equation
of (\ref{IsoFlocGen}), it yields:
\begin{equation}                              \label{equ}
  u^*=U(s^*) \quad \mbox{with} \quad 
  U(s)=\frac{\varphi_u(s)(\varphi_v(s)-b)}{a[\varphi_v(s)-\varphi_u(s)]}.
\end{equation}
Note that $u^*$ defined by (\ref{equ}) is positive because $s^*\in
I$. By replacing $u^*$ by (\ref{equ}) in (\ref{eqV-U}), we get:
\begin{equation}                              \label{eqv}
  v^*=V(s^*) \quad \mbox{with} \quad 
  V(s)=-\frac{\varphi_u^2(s)(\varphi_v(s)-b)}{a[\varphi_v(s)-\varphi_u(s)]\varphi_v(s)}.
\end{equation}
Substituting the expressions of $U(s^*)$ and $V(s^*)$ given by
(\ref{equ}) and (\ref{eqv}) in
the expression of  $H(s^*)$ yields a characterization of $s^*$:
\begin{equation}                           \label{eqH(S)}
  D(S_{in}-s^*)=H(s^*)\quad \mbox{with} \quad
  H(s)=D\frac{\varphi_u(s)(\varphi_v(s)-b)}{a\varphi_v(s)} .
\end{equation}
Note that for all $s\in I$ , $U(s)>0$, $V(s)>0$ and $H(s)>0$  and that :
$$\lim_{s\to\lambda_u}H(s)=0,\quad \lim_{s\to\lambda_v}H(s)=+\infty .$$
In addition, function $H$ is strictly increasing on $I$. Indeed, we have:
$$H'(s)=\frac{D}{a}
\frac{\varphi_v(s)(\varphi_v(s)-b)\varphi_u'(s)+b\varphi_u(s)\varphi_v'(s)}{\varphi_v^2(s)}>0 .$$
Consequently, Equation (\ref{eqH(S)}) admits a unique solution 
$s^*\in I=]\lambda_u,\lambda_v[$ if and
only if $S_{in}>\lambda_u$, which is equivalent to $\mu_u(S_{in})>D$.
\end{proof}

\subsection{Study of stability}
Under the conditions of stability and global attractiveness of the
washout steady-state of the chemostat model in which only the
planktonic biomass would be considered (see \cite{SW95,HLRS17}):
\begin{equation}
  \label{condlessivage}
  D \geq \mu_{u}(S_{in})
\end{equation}
one can easily check that the washout $(S_{in},0,0)$ is also the only steady-state of the
system (\ref{chem_attach}), stable
and globally attractive. As a matter of fact, by considering the reduced model (\ref{chem_attach_reduit}),
under this assumption we have:
\[
x \in ]0,S_{in}] \; \Rightarrow \;
\frac{dx}{dt}=(\mu_{u}(S_{in}-x)-D)u+(\mu_{v}(S_{in}-x)-D)v
<0
\]
which demonstrates that $x(\cdot)$ asymptotically converges towards
$0$ for any initial condition. 
As any solution of system 
(\ref{chem_attach}) is bounded, we deduce that it converges to the washout equilibrium.
According to the study conducted in Section \ref{sec_coex_agreg}, a positive steady-state exists
as soon as the condition (\ref{condexist}) is verified and is unique. By particularizing the attachment
and detachment functions as we did in Section \ref{sec_coex_agreg}, the following stability
result is obtained (the case in which $D_{u}$ and $D_{v}$ are
different from $D$ is addressed in \cite{FHCRS13}).
\begin{proposition}
  Under the assumptions of Proposition \ref{propexist} the coexistence
  steady-state is a locally exponentially stable of system (\ref{chem_attach}).
\end{proposition}

\begin{proof}
As mentioned previously, it is enough to study the local
  stability for the reduced dynamics (\ref{chem_attach_reduit}).
  The Jacobian matrix of (\ref{chem_attach_reduit}) for the steady-state $(u^*,v^*)$, which corresponds
  to the positive equilibrium $E^*=(s^*,u^*,v^*)$ of (\ref{chem_attach}), is equal to:
  $$
  J^*=
  \left[
    \begin{array}{ll}
      -u^*\varphi_u'(s^*)+\varphi_u(s^*)-a(2u^*+v^*)  & -u^*\varphi_u'(s^*)-au^*+b  \\[2mm]
      -v^*\varphi_v'(s^*)+a(2u^*+v^*)  &-v^*\varphi_v'(s^*)+\varphi_v(s^*)+au^*-b 
    \end{array}
  \right]
  $$
  The trace of this matrix is equal to:
  $${\rm Tr}J^*=-u^*\varphi_u'(s^*)-v^*\varphi_v'(s^*)+\varphi_u(s^*)-a(u^*+v^*)+\varphi_v(s^*)-b$$
  Note that based on Equations (\ref{IsoFlocGen}), it can be deduced
  that:
  \begin{equation}
    \varphi_u(s^*)-a(u^*+v^*)=-b\frac{v^*}{u^*}<0,\qquad \varphi_v(s^*)-b=-a\frac{(u^*+v^*)u^*}{v^*}<0
    \label{equv}
  \end{equation}
  Further, as $\varphi_u'(s^*)>0$ and $\varphi_v'(s^*)>0$,
  it can be deduced that ${\rm Tr}J^*<0$. 
  The determinant of this matrix is equal to:
  $${\rm Det}J^*=Au^*\varphi_u'(s^*)+Bv^*\varphi_v'(s^*)+C$$
  with:
  $$
  A=a(u^*+v^*)+b-\varphi_v(s^*),
  \quad
  B=a(u^*+v^*)+b-\varphi_u(s^*),
  $$
  and:
  $$
  C=\varphi_u(s^*)\varphi_v(s^*)+\varphi_u(s^*)(au^*-b)-\varphi_v(s^*)a(2u^*+v^*)
  $$
  By using Expressions (\ref{equv}), it yields that:
  $$
  A=a\frac{(u^*+v^*)^2}{v^*}>0,
  \quad
  B=b\frac{u^*+v^*}{u^*}>0
  $$
  Moreover, we have:
  $$
  C=\varphi_u(s^*)\left(\varphi_v(s^*)-b\right)+a\left(u^*\varphi_u(s^*)-v^*\varphi_v(s^*)\right)-2au^*\varphi_v(s^*)
  $$
  Utilizing (\ref{eqPhiPsi}), we get:
  $$
  C=\varphi_u(s^*)\left(\varphi_v(s^*)-b\right)+
  2au^*\varphi_u(s^*)-2au^*\varphi_v(s^*)
  $$
  Utilizing (\ref{equv}), we have:
  $$
  au^*\left(\varphi_u(s^*)-\varphi_v(s^*)\right)=-\varphi_u(s^*)\left(\varphi_v(s^*)-b\right)$$
  Consequently:
  $$
  C=-\varphi_u(s^*)\left(\varphi_v(s^*)-b\right)>0
  $$
  Thereof, it can be deduced that ${\rm Det}J^*>0$, and as a
  consequence, the real parts of the eigenvalues of $J^*$ are strictly negative.
\end{proof} 

\section{The case of fast attachments/detachments}
\label{SecFlocsLentRapide}

Depending on species and on hydrodynamic conditions, attachment and detachment
velocities may prove to be large compared to growth kinetics and to dilution
rate. In this case, it is possible to consider that the attachment and detachment terms,
$\alpha(\cdot)$ and $\beta(\cdot)$ respectively, can be rewritten in
the form:
\[
\frac{\alpha(\cdot)}{\varepsilon}, \quad
\frac{\beta(\cdot)}{\varepsilon}
\]
where $\varepsilon$ is a positive number supposed to be small, and
functions $\alpha(\cdot)$, $\beta(\cdot)$ verify the
same assumptions \ref{hypagreg}. Thus, the model (\ref{chem_attach}) is written as:
\begin{equation}
  \label{chem_attachLR}
  \left\{
  \begin{array}{lll}
    \ds \frac{ds}{dt} & = & \ds D(S_{in}-s)-\mu_{u}(s)u-\mu_{v}(s)v\\[4mm]
    \ds \frac{du}{dt} & =  & \ds \mu_{u}(s)u-Du
    -\frac{1}{\epsilon}\left(\alpha(u,v)u-\beta(v)v\right)\\[4mm]
    \ds \frac{dv}{dt} & =  & \ds \mu_{v}(s)v-Dv
    +\frac{1}{\epsilon}\left(\alpha(u,v)u-\beta(v)v\right)
  \end{array}\right.
\end{equation}
It is convenient to write this dynamic by replacing the variables $u$
and $v$ by $x=u+v$ and $p=u/x$
\begin{equation}
  \label{chem_attach_xp}
  \left\{
  \begin{array}{lll}
    \ds \frac{ds}{dt} & =  & \ds D(S_{in}-s)-\bar\mu(s,p)x\\[4mm]
    \ds \frac{dx}{dt} & =  & \ds \bar\mu(s,p)x-Dx\\[4mm]
    \ds \frac{dp}{dt} & =  & \ds
    \left(\mu_{u}(s)-\mu_{v}(s)\right)p(1-p)-\frac{1}{\epsilon}\left(\alpha(px,(1-p)x)p-\beta((1-p)x)(1-p)\right)
  \end{array}\right.
\end{equation}
by defining:
\[
\bar\mu(s,p):=p\,\mu_{u}(s)+(1-p)\,\mu_{v}(s)   .
\]
Observe that this dynamic system is of the form:
\[
\left\{
\begin{array}{lll}
  \ds \frac{ds}{dt} & =  & \ds f_{s}(s,x,p)\\[4mm]
  \ds \frac{dx}{dt} & =  & \ds f_{x}(s,x,p)\\[4mm]
  \ds \frac{dp}{dt} & =  & \ds \frac{1}{\epsilon}\left[\epsilon f_{p}(s,p)+g(x,p)\right]
\end{array}\right.
\]
where we posit:
\[
g(x,p):=-\alpha(px,(1-p)x)p+\beta((1-p)x)(1-p) .
\]
When $\epsilon$ is small and the terms $f_{s}(s,x,p)$, $f_{x}(s,x,p)$
and $\epsilon f_{p}(s,p)+g(x,p)$ are
of the same order of magnitude, the velocity $\frac{dp}{dt}$
is then very large compared to velocities $\frac{ds}{dt}$, $\frac{dx}{dt}$. 
Variables $s$ and $x$ can then be considered as almost constant and the
approximation of the dynamics of variable $p$ as "fast":
\begin{equation}
  \label{dyn_p_reduit}
  \frac{dp}{dt} = \frac{1}{\epsilon}g(x,p)
\end{equation}
where $s$ is considered as a constant parameter (the term $\epsilon f_{p}(s,p)$ being negligible
with regard to $g(x,p)$). If for any $x$, the differential equation (\ref{dyn_p_reduit}) admits a unique
steady-state $\bar p(x)$, then this expression can be carried to
the system (\ref{chem_attach_xp}) to obtain the "slow" approximation of the dynamics of the variables
$s$ and $x$:
\begin{equation}
  \label{chem_attach_red}
  \left\{
  \begin{array}{lll}
    \ds \frac{ds}{dt} & =  & \ds D(S_{in}-s)-\mu(s,x)x\\[4mm]
    \ds \frac{dx}{dt} & =  & \ds \mu(s,x)x-Dx
  \end{array}\right.
\end{equation}
by defining:
\[
\mu(s,x)=\bar\mu(s,\bar p(x)) .
\]
This reduction technique (which consists in replacing $\epsilon$ by $0$) is well known in
physics under the name of quasi-steady state approximation method. At the mathematical
level, the rigorous proof of the convergence of the solutions of the system
(\ref{chem_attach_xp}) towards those of the reduced system (\ref{chem_attach_red}) makes use of the theory of singular
perturbations (see for instance \cite{K96}).
When the slow manifold is globally attractive, that is when $\bar
p(x)$ is a globally asymptotically stable of the dynamics
$dp/d\tau=g(x,p)$ for any fixed $x>0$ (where $\tau=t/\epsilon$ is the
``fast'' time), then Tikhonov's Theorem
applies.
Recall that this Theorem asserts that for any initial condition of 
(\ref{chem_attach_xp}) with $x(0)>0$ and any time interval $[0,T]$ with $T>0$,
the solution $s(\cdot)$, $x(\cdot)$ of (\ref{chem_attach_xp}) converge uniformly on
$[0,T]$ to the solution of (\ref{chem_attach_red}). Furthermore, when the solution of
the reduced dynamics (\ref{chem_attach_red}) converges to an asymptotically stable
equilibrium, then one can take $T=+\infty$ (see for instance \cite{LST98}).
The Proposition below shows that the existence and the global asymptotic
stability of the slow manifold, under
Assumptions \ref{hypagreg}. 

\begin{proposition}
\label{propbarp}
  Under Assumptions \ref{hypagreg}, there exists a unique function 
  $\bar p: \Rset_{+} \mapsto [0,1]$ $C^1$, strictly decreasing, such
  that $g(x,\bar p(x))=0$ for all $x>0$. In
  addition, $\bar p(x)$ is the unique globally asymptotically stable
  steady-state of the scalar equation (\ref{dyn_p_reduit}), for all $x>0$.
\end{proposition}

\begin{proof}
  For any $x>0$, we have $g(x,0)=\beta(x)>0$ and $g(x,1)=-\alpha(x,0)<0$
  (following Assumptions \ref{hypagreg}). According to the intermediate value theorem,
  there therefore exists $\bar p(x) \in ]0,1[$ such that $g(x,\bar p(x))=0$. Let us determine the
  partial derivatives of the function $g$:
  \[
  \begin{array}{lll}
    \ds \frac{\partial g}{\partial x} & = &
    \ds
    -\left[\left(\frac{\partial \alpha}{\partial
          u}(u,v)p+\frac{\partial\alpha}{\partial v}(u,v)(1-p)\right)p-\beta'(v)(1-p)^2\right]_{u=px,v=(1-p)x}\\[5mm]
    \ds \frac{\partial g}{\partial p} & = &
    \ds -\left[\left(\frac{\partial\alpha}{\partial
          u}(u,v)-\frac{\partial\alpha}{\partial v}(u,v)\right)u+\alpha(u,v)+\frac{1}{u+v}\frac{d}{dv}(\beta(v)v)\right]_{u=px,v=(1-p)x}
  \end{array}
  \]
  For $x>0$, Assumptions \ref{hypagreg} guarantee $\frac{\partial
    g}{\partial x}<0$ and $\frac{\partial g}{\partial p}<0$.
  Thus, the function
  $p \mapsto g(x,p)$  is strictly decreasing, guaranteeing the uniqueness of the solution $\bar p(x)$
  of $g(x,p)=0$. According to the implicit function theorem, the function $\bar p$ is also
  differentiable for any $x>0$ and its derivative is written as:
  \[
  \bar p'(x)=-\frac{\ds\frac{\partial g}{\partial x}(x,\bar p(x))}{\ds\frac{\partial g}{\partial
      p}(x,\bar p(x))} <0 .
  \]
  The function $\bar p$ is thus $C^1$ on $\Rset_{+}\setminus\{0\}$ and strictly decreasing.
  Thereby, for all fixed $x>0$, $\bar p(x)$ is the unique steady-state of the differential
  equation (\ref{dyn_p_reduit}), and since $\frac{\partial g}{\partial
    p}<0$ for every $(x,p)$, it can be thereof deduced that the
  steady-state $\bar p(x)$ is globally asymptotically stable for the
  scalar dynamics (\ref{dyn_p_reduit}).
\end{proof}

For instance, for functions considered in (\ref{alpha_beta_simples}),
we get:
\begin{equation}
\label{exp}
\bar p(x)=\frac{1}{\ds 1+\frac{a}{b}x} .
\end{equation}

Figure \ref{figSF} presents simulations with functions
(\ref{alpha_beta_simples}) and compares the solutions (in plain line)
of the original system (\ref{chem_attach_xp}) with the ones (in
dashed line) of the reduced dynamics (\ref{chem_attach_red}).
It shows that the slow-fast approximation is good even for value of
$\epsilon$ that are not so small.

\begin{figure}[h!]
\begin{center}
\begin{tabular}{ll}
\begin{minipage}{6cm}
\includegraphics[width=7cm,height=6cm]{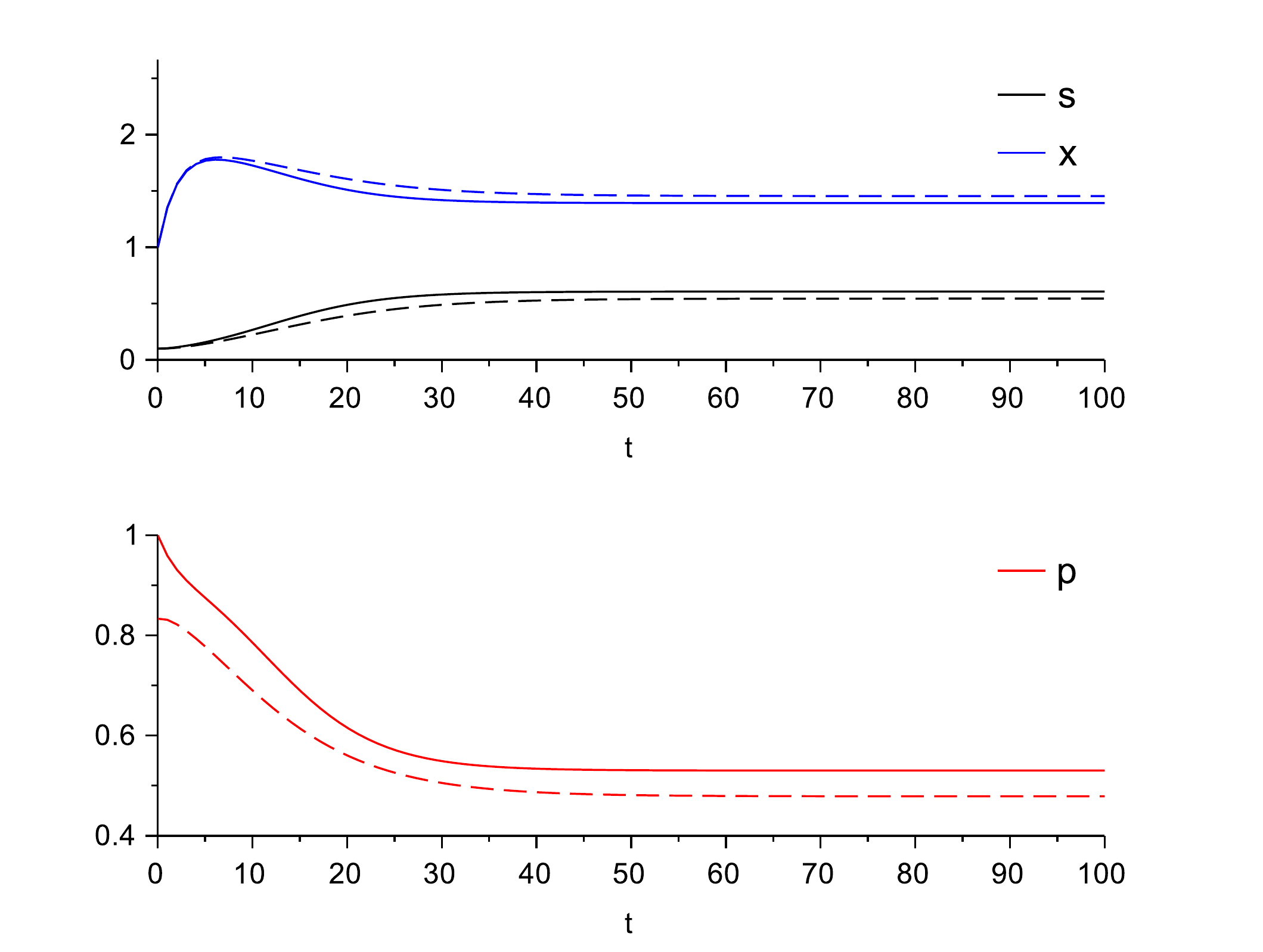} 
\end{minipage}
&
\begin{minipage}{6cm}
\includegraphics[width=7cm,height=6cm]{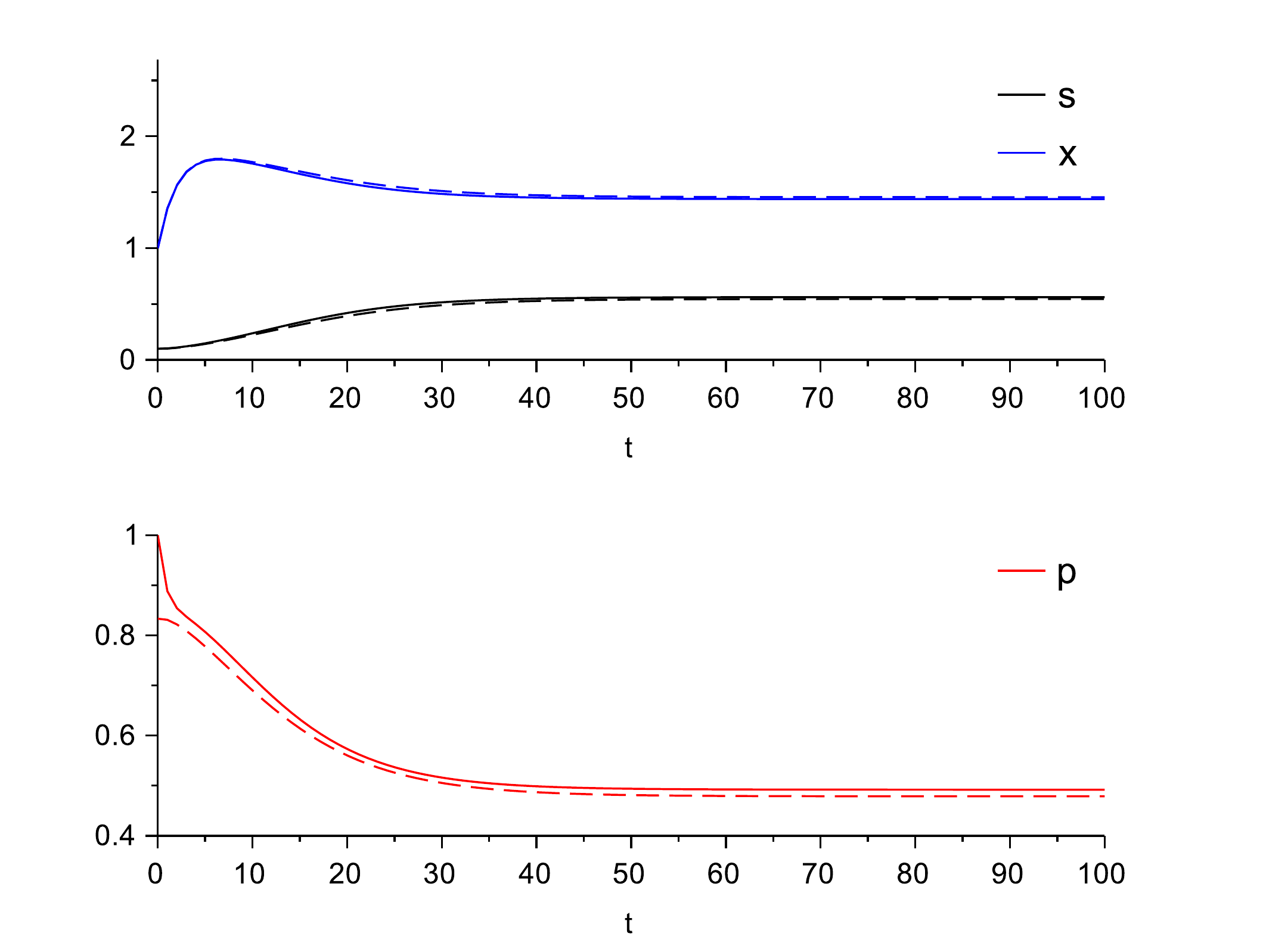}
\end{minipage}
\end{tabular}
\end{center}
\caption{\label{figSF}
Simulations for $\mu_{u}(s)=\frac{s}{1+s}$,
$\mu_{v}(s)=\frac{0.7s}{1+s}$, $S_{in}=2$, $D=0.5$, $a=1$, $b=0.5$
with $\epsilon=2$ (left) and $\epsilon=0.5$ (right)
}
\end{figure}

\begin{remark}
  Thanks to Assumptions \ref{hypagreg}, it yields that:
  \[
  \frac{\partial \mu}{\partial x}(s,x)= \frac{\partial \bar\mu}{\partial
    p}(s,p)\vert_{p=\bar p(x)}.\bar p'(x)=(\mu_{u}(s)-\mu_{v}(s)).\bar p'(x)<0
  \]
  and thus the model (\ref{chem_attach_red}) for the total biomass $x$
  has a density-dependent growth, decreasing with respect $x$.
\end{remark}
\subsection{Consideration of several species}
When several species are in competition, we can similarly decompose the biomass
of each species $i$ into planktonic biomass $u_{i}$ and attached biomass $v_{i}$ (without differentiating
the composition of flocs which can mix individuals from different
species):
\[
  \left\{
  \begin{array}{lll}
    \ds \frac{ds}{dt} = & \ds D(S_{in}-s)-\sum_{j=1}^n\mu_{u_{j}}(s)u_{j}-\sum_{j=1}^n\mu_{v_{j}}(s)v_{j}\\[4mm]
    \ds \frac{du_{i}}{dt} =  & \mu_{u_{i}}(s)u_{i}-Du_{i}
-\alpha_{i}(u_1,\cdots,u_{n},v_{1},\cdots,v_{n})u_{i}+\beta(v_{1},\cdots,v_{n})v_{i}\\
& & \hfill (i=1\cdots n)\\
    \ds \frac{dv_{i}}{dt} =  & \mu_{v_{i}}(s)v_{i}-Dv_{i}
+\alpha_{i}(u_1,\cdots,u_{n},v_{1},\cdots,v_{n})u_{i}-\beta(v_{1},\cdots,v_{n})v_{i}
\end{array}\right.
\]
The specific attachment functions $\alpha_{i}$ then depend (a priori) on all others quantities $u_{j}$, $v_{j}$  since a free individual of species $i$ can attach to free biomass or biomass with any
species attached. Analogously, the specific detachment functions $\beta_{i}$ depend a priori
on all quantities $v_{j}$ of biomass attached where an individual $i$
could have attached.
To simplify, it will be possible, for example, to assume that
the $\alpha_{i}$ are functions of the total
planktonic and attached biomass $u=\sum_{j}u_{j}$ and $v=\sum_{v}v_{j}$,
and the $\beta_{i}$ functions
of $v$ only, with the same Assumptions (\ref{hypagreg}). The combinatorics of the possible specific
cases makes the mathematical study much more complicated, but when the attachment
and detachment velocities can be considered to be fast, the quasi-steady state approximation
makes it possible to write a dynamic system for biomass $x_{i}=u_{i}+v_{i}$ by
expressing the terms $u_{i}$ and $v_{i}$ according to all the $x_{j}$ on the "slow" manifold defined by
the system of equations:
\[
\alpha_{i}(u_1,\cdots,u_{n},v_{1},\cdots,v_{n})u_{i}-\beta_{i}(v_{1},\cdots,v_{n})v_{i}=0
\qquad i=1\cdots n .
\]
For example, by considering simple functions like we did in
(\ref{alpha_beta_simples}):
\[
\alpha_{i}(x_{1},\cdots,x_{n})=\sum_{j=1}^na_{ij}x_{j}, \quad \beta_{i}=b_{i}
\]
where parameters $a_{ij}$ reflect how easily an individual of species $i$ attaches to an
individual of species $j$, the following expressions are obtained for
the proportions
$q_{i}=u_{i}/x_{i}$ on the slow manifold, which is uniquely
  defined by
\[
\bar q_{i}(x_{1},\cdots,x_{n})=\frac{1}{\ds 1+\frac{1}{b_{i}}\sum_{j=1}^na_{ij}x_{j}}
\]
as in Section \ref{SecFlocsLentRapide} (under the assumption of fast attachments and detachments), and
the reduced system is then written as:
\[
\left\{\begin{array}{lll}
\ds \frac{ds}{dt} & = & \ds D(S_{in}-s)-\sum_{j=1}^n\mu_{j}(s,x_{1},\cdots,x_{n})x_{j}\\[4mm]
\ds \frac{dx_{i}}{dt} & = & \ds
\mu_{i}\left(s,x_{1},\cdots,x_{n}\right)x_{i}-Dx_{i} \qquad (i=1\cdots
n)
\end{array}\right.
\]
by setting:
\[
\mu_{i}(s,x)=\bar q_{i}(x_{1},\cdots,x_{n})\mu_{u_{i}}(s)+(1-\bar q_{i}(x_{1},\cdots,x_{n}))\mu_{v_{i}}(s)
\]
The dynamics of the fast variables $q_{i}$ is given by the system
\[
\frac{dq_{i}}{d\tau}=-\alpha_{i}(x)q_{i}+b_{i}(1-q_{i}) \qquad
(i=1\cdots n)
\]
(where $\tau=t/\epsilon)$ for which $(\bar q_{1},\cdots,\bar q_{n})$ is clearly the unique
globally asymptotically stable equilibrium, for any fixed
$(x_{1},\cdots, x_{n})$. Therefore Thikonov's Theorem applies.
Notice that $\mu_{i}$ are density-dependent growth functions,
decreasing with respect to the $x_{i}$. This then
exactly corresponds to the context of density-dependent competition model, which shows that a coexistence
between species is possible \cite{LMR05,ADLS06}. 
It is thus concluded that a mechanism of (fast)
attachment and detachment of biomass is a possible (theoretical) explanation for the
maintaining of biodiversity in a chemostat.

\section{Consideration of distinct removal rates}
\label{SectionDiffD}

In this Section, we consider that the removal rates of planktonic and
attached bacteria are distinct, and accordingly to Assumptions
(\ref{hypagreg}) one has $D_{v}<D_{u}\leq D$.
This Section follows part of the work \cite{F13,FHCRS13}.
The reduction technique we use in Section
\ref{SecFlocsLentRapide} gives the following reduced model:
\begin{equation}
  \label{chem_attach_red_D}
  \left\{
  \begin{array}{lll}
    \ds \frac{ds}{dt} & =  & \ds D(S_{in}-s)-\mu(s,x)x\\[4mm]
    \ds \frac{dx}{dt} & =  & \ds \mu(s,x)x-d(x)x
  \end{array}\right.
\end{equation}
where we posit:
\[
d(x)=\bar p(x)D_{u} +(1-\bar p(x))D_{v} .
\]
Notice that the dynamics of the fast variable $p$ is given by
  equation (\ref{dyn_p_reduit}), exactly as in Section \ref{SecFlocsLentRapide}.
Therefore, Proposition \ref{propbarp} applies.
Let us underline that having a density dependent removal rate in the
chemostat model has not being considered (and justified) before in the
literature. 

\medskip

As in Section \ref{sec_coex_agreg}, we consider break-even
concentrations $\lambda_{u}$, $\lambda_{v}$ associated to functions
$\mu_{u}$ and $\mu_{v}$ but here for the distinct removal rates
$D_{u}$, $D_{v}$ (which are numbers that verify
$\mu_{u}(\lambda_{u})=D_{u}$ and $\mu_{v}(\lambda_{v})=D_{v}$).
Differently to the case of identical
removal rates, for which Assumptions \ref{hypagreg} implies the inequality
$\lambda_{u}<\lambda_{v}$, this later inequality is no longer
necessarily satisfied, as depicted on Figure \ref{figlambdas}.
\begin{figure}[h!]
  \begin{center}
   \includegraphics[width=.4\textwidth]{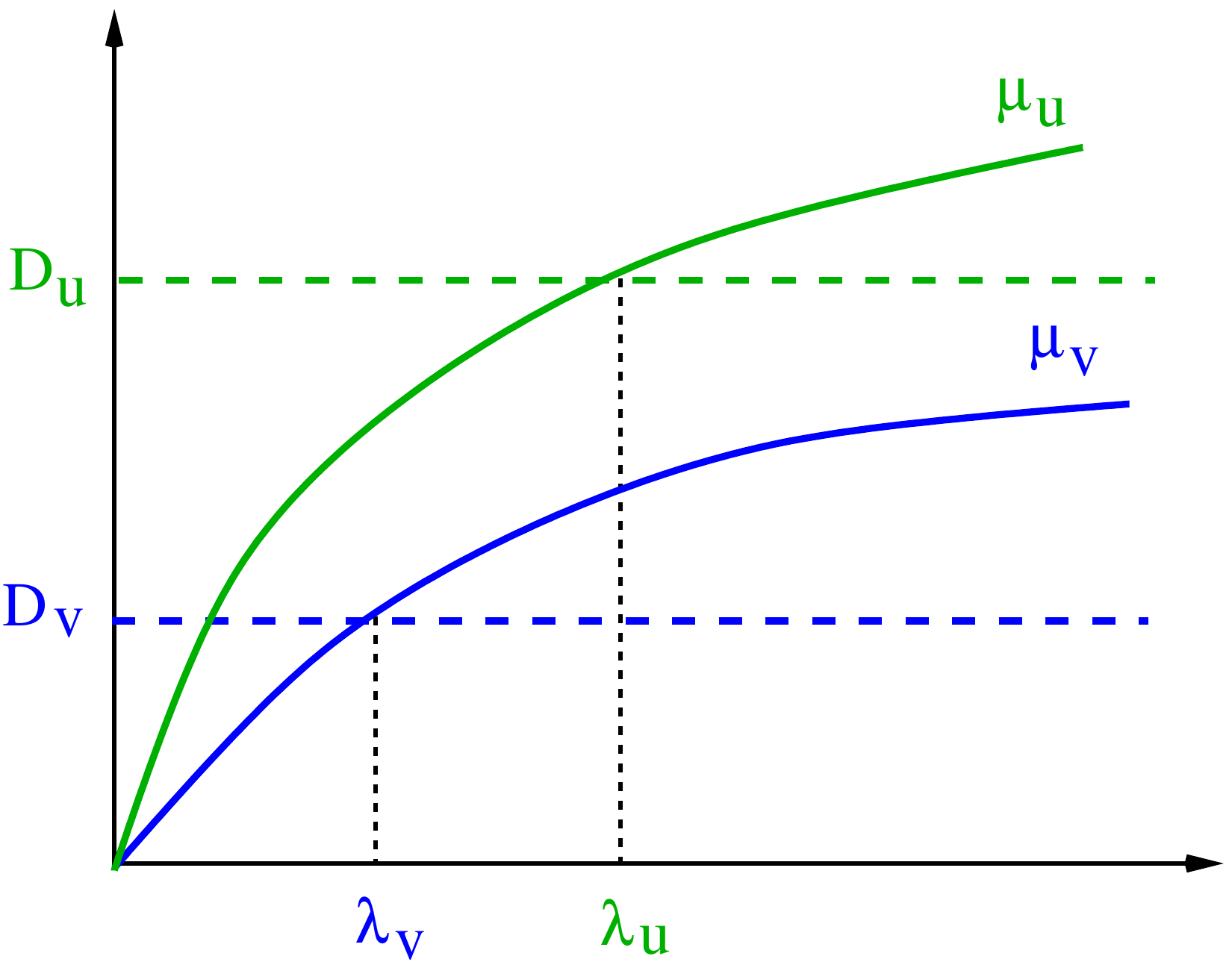}
\hspace{2mm}
   \includegraphics[width=.4\textwidth]{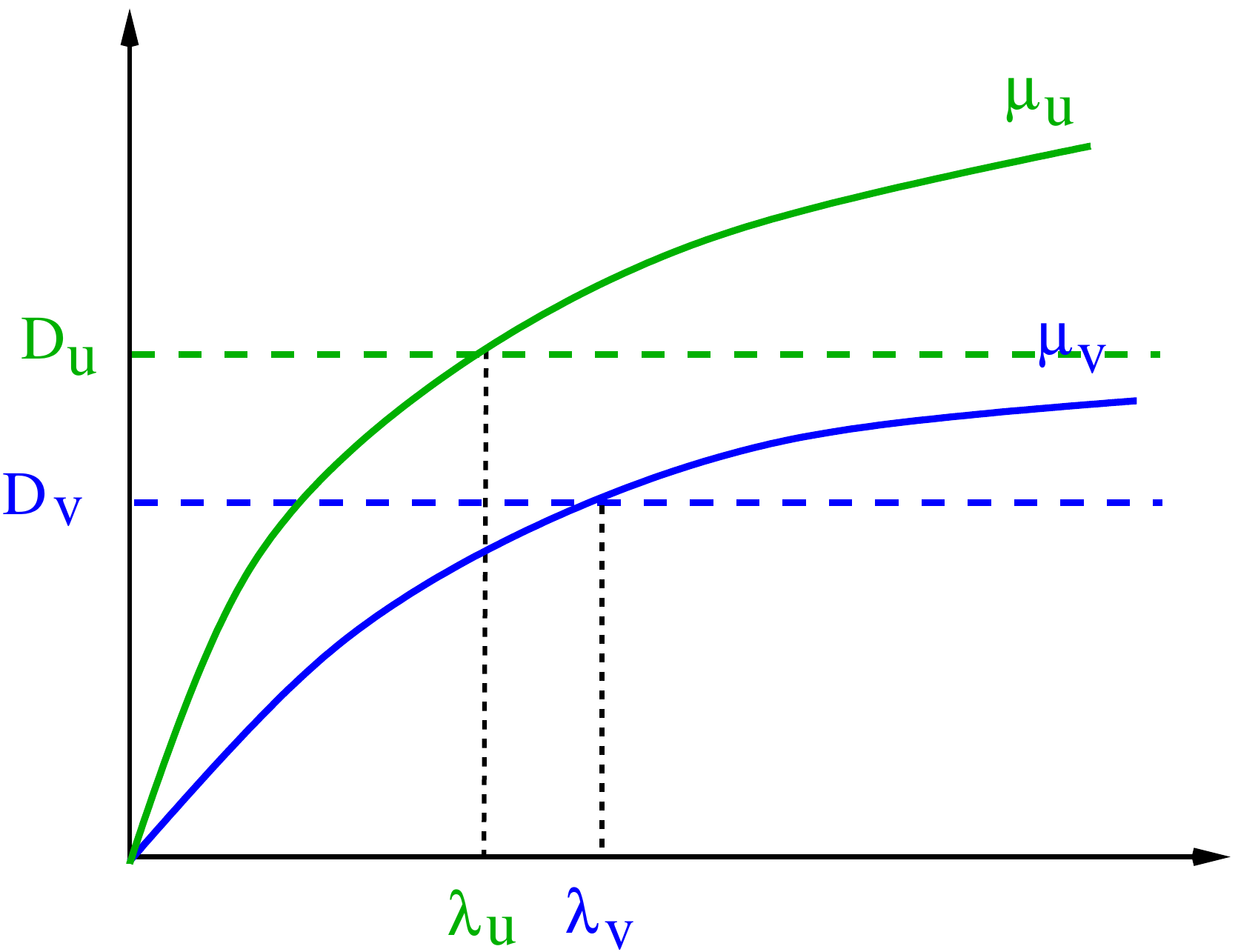}
    \caption{One can have $\lambda_{u}>\lambda_{v}$ (left) as well as
      $\lambda_{u} <\lambda_{v}$ (right).
      \label{figlambdas}}
  \end{center}
\end{figure}

\medskip
The model (\ref{chem_attach_red_D}) admits clearly the washout $(S_{in},0)$ as
an equilibrium, and let us study the possibility for the system to have
another steady state.
A positive equilibrium $(s^\star,x^\star)$ of dynamics (\ref{chem_attach_red_D})
has to fulfill
\begin{equation}
s^\star = \gamma(x^\star) := S_{in} -\frac{x^\star d(x^\star)}{D}
\end{equation}
and 
\begin{equation}
\label{equmud}
\mu(s^\star,x^\star)=d(x^\star)
\end{equation}
Notice that when $s<\min(\lambda_{u},\lambda_{v})$,
resp. $s>\max(\lambda_{u},\lambda_{v})$, one has $\mu(s,x)<d(x)$,
resp. $\mu(s,x)>d(x)$,  for any $x$. Therefore, one has
\[
s^\star \in
[\min(\lambda_{u},\lambda_{v}),\max(\lambda_{u},\lambda_{v})] .
\]
Since the functions $\mu_{u}$ and $\mu_{v}$ are increasing, the map $s
\mapsto \mu(s,x)$ is increasing for any $x$ and by the Implicit
Function Theorem, we deduce the existence of an unique solution of
(\ref{equmud}) as $s^\star=\phi(x^\star)$.
Therefore, a positive equilibrium (if it exists) has to fulfill
\[
\Gamma(x^\star):=\gamma(x^\star)-\phi(x^\star)=0 .
\]
Notice that one has $\Gamma(0)=S_{in}-\lambda_{u}$ and
$\Gamma(+\infty)=-\infty$. Therefore, the existence of a positive
equilibrium is guaranteed when $\lambda_{u}<S_{in}$.
Notice that this last condition is exactly the one that guarantees the
existence of a positive equilibrium for the chemostat model without
attachment:
\[
\left\{\begin{array}{lll}
    \ds \frac{ds}{dt} & =  & \ds D(S_{in}-s)-\mu_{u}(s)u\\[4mm]
    \ds \frac{du}{dt} & =  & \ds \mu_{u}(s)u-D_{u}u
  \end{array}\right.
\]
We examine now the possibilities of having more than one positive equilibrium. 
The function $\gamma$ is such that $\gamma(0)=S_{in}$ and
$\gamma(+\infty)=-\infty$. So, it has to decrease somewhere on the
interval $[0,+\infty)$.
From the Implicit Function Theorem, we can write
\[
\phi'(x) =
\frac{d'(x)-\frac{\partial\mu}{\partial
    x}(\phi(x),x)}{\frac{\partial\mu}{\partial s}(\phi(x),x)}
= \frac{\bar
  p'(x)}{\frac{\partial\mu}{\partial s}(\phi(x),x)}\left(D_{u}-D_{v}-\mu_{u}(\phi(x))+\mu_{v}(\phi(x))\right)
\]
When $\lambda_{u}<\lambda_{v}$, one has $\mu_{u}(s)\geq D_{u}$ and
$\mu_{v}(s)<D_{v}$ for any $s \in [\lambda_{u},\lambda_{v})$. 
As $\bar p'(x)<0$ (see Proposition \ref{propbarp}) and 
$\frac{\partial\mu}{\partial s}(\phi(x),x)>0$, we deduce
$\phi'(x)>0$ for any $x$ such that $\phi(x) \in
[\lambda_{u},\lambda_{v})$. 
At the opposite, when $\lambda_{u}>\lambda_{v}$, one has $\phi'(x)<0$
for any $x$ such that $\phi(x) \in [\lambda_{v},\lambda_{u})$.
This leaves open the possibility of having the functions $\gamma$ and $\phi$
simultaneously decreasing with more than one intersection of their
graphs (and then having the function $\Gamma$ non-monotonic with
alternate signs of $\Gamma'(x^\star)$ at the solutions $x^\star$).
At a positive equilibrium $E^*=(s^*,x^*)$, the Jacobian matrix is:
\[
J(E^*)=\left[\begin{array}{cc}
\ds -D-x^*\frac{\partial \mu}{\partial s}(s^*,x^*) & \ds -x^*\frac{\partial
  \mu}{\partial x}(s^*,x^*)-d(x^*)\\[4mm]
\ds x^*\frac{\partial \mu}{\partial s}(s^*,x^*) & \ds x^*\frac{\partial
  \mu}{\partial x}(s^*,x^*) -x^*d'(x^*)
\end{array}\right]
\]
with determinant:
\[
det J(E^*)=Dx^*\left(d'(x^*)-\frac{\partial
  \mu}{\partial x}(s^*,x^*)\right)+x^*\frac{\partial \mu}{\partial
s}(s^*,x^*)
\frac{d}{dx}[xd(x)](x^*) .
\]
One can easily check that it can be also written as
\[
det J(E^*)=-Dx^*\frac{\partial \mu}{\partial x}(s^*,x^*)\Gamma'(x^*) 
\]
which shows an alternation of stability of the equilibriums $E^*$
depending on the sign of $\Gamma'(x^*)$.
We illustrate the possibility of having multiple-stability in the case 
$\lambda_{v}<\lambda_{u}<S_{in}$ with the functions $\alpha$, $\beta$ given in
  (\ref{alpha_beta_simples}), that provide the simple expression
  (\ref{exp}) of the function $\bar p(\cdot)$, and Monod expressions
  for functions $\mu_{u}$, $\mu_{v}$.
Even in this simple case, the expression of the function $\Gamma$
is too complicated to conduct an analytic study. Figure
\ref{figbistability} presents the phase portrait of the reduced dynamics
(\ref{chem_attach_red_D}) and shows its bi-stability for the numerical values of the parameters
that have been chosen.
\begin{figure}[h!]
\begin{center}
\includegraphics[width=8.5cm]{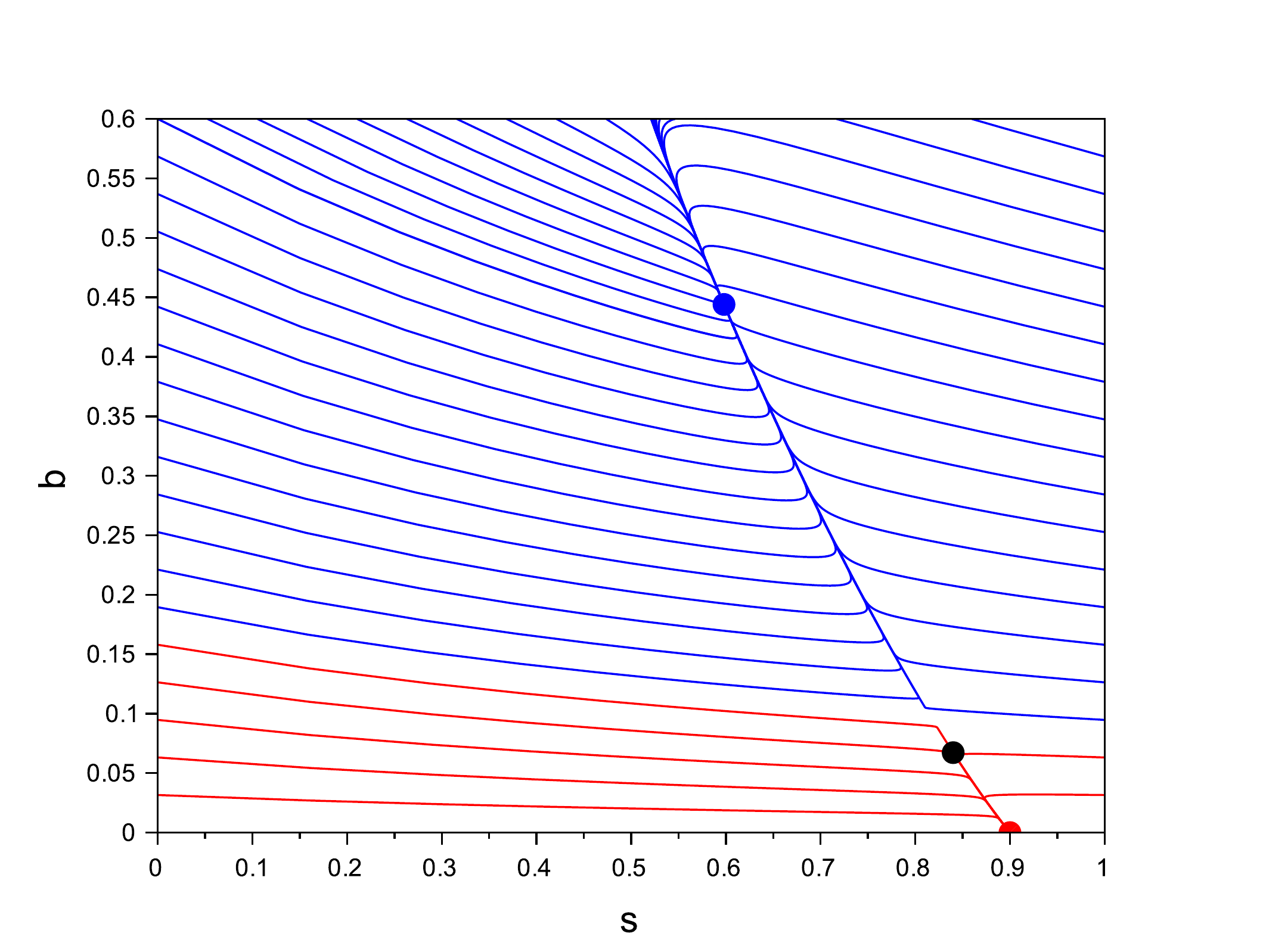}
\caption{\label{figbistability} Example of bi-stability with 
$\mu_{u}(s)=\frac{2s}{1+s}$, $\mu_{v}(s)=\frac{1.5 s}{0.8+s}$,
$D_{u}=1$, $D_{v}=0.5$,  $S_{in}=0.9$, $D=1$, $a/b=4$.}
\end{center}
\end{figure}

In the reference \cite{F13}, it is shown that under the additional assumption that the
map $x^* \mapsto x^*\bar p(x^*)$ is increasing, the multiplicity can
indeed occur only when $\lambda_{u}>\lambda_{v}$, and that generically
each equilibrium is necessarily either a stable node or a saddle
point.
Therefore, Tikhonov's Theorem, that has been recalled in Section
  \ref{SecFlocsLentRapide}, allows to claim that for any initial condition of the system
  (\ref{chem_attach}) such that
  $(s_{0},x_{0})$ does not belong to the stable manifold of a saddle
  equilibrium of the reduced dynamics (\ref{chem_attach_red_D}), the solution  $s(\cdot)$,
  $x(\cdot)$ converges to the solution of the reduced dynamics on the
  $[0,+\infty)$ time interval, that is for almost any initial condition.

\medskip
Finally, this shows that multiple stability can occur in the chemostat
model with attachment and distinct removal rates, even though the
growth functions are monotonically increasing.
This fact is quite remarkable comparing to the classical chemostat
model (i.e. without attachment) for which a multiple stability is
possible only for non-monotonic growth functions (see for instance
\cite{HLRS17}).
Nevertheless, the analysis of all the generic behaviors of the solutions of the model with
several species (and different removal rates) remains today an open
problem. Dynamics in dimension higher than two potentially reserve a richness
of possible behaviors. In particular, the possibility of having
unstable nodes leave open the possibilities of having limit cycles, as
illustrated in \cite{FRS16}.

\section{Conclusion}

In this work we have proposed a generic framework of chemostat models
with free and attached biomass compartments. Under
general assumptions, we have shown that a coexistence of the two forms
is possible and leads to a unique positive equilibrium which is
moreover globally asymptotically stable.
When the assumptions about fast attachment and detachment are
justified, we have also shown that reduced
models with the total biomass instead of planktonic and
attached ones provide natural extensions of the classical chemostat
model with a density-dependent growth function,
such as in the Contois model \cite{C59}. This allows coexistence of multiple
species when each of them can be present in the two forms: planktonic
and attached (with same or different species).
We have also shown that the consideration of different removal rates
for the free and attached biomass could lead to some non-intuitive
behaviors, such as multiple stability, that is today widely not well
understood in presence of several species.

\begin{acknowledgement}
This work has been initiated in the ``DISCO'' project funded by the
French National Research Agency (ANR) in the SYSCOMM program.
The author warmly thanks T. Sari, C. Lobry, J. Harmand and R. Fekih-Salem,
whose PhD work having inspired the present paper.
\end{acknowledgement} 

\bibliographystyle{plain} 
\bibliography{biblio}

\end{document}

%% file: ejmcommands.tex
\usepackage{hyperref}

\let\realverbatim=\verbatim
\let\realendverbatim=\endverbatim
\renewcommand\verbatim{\par\addvspace{6pt plus 2pt minus 1pt}\realverbatim}
\renewcommand\endverbatim{\realendverbatim\addvspace{6pt plus 2pt minus 1pt}}

\ifprodtf \else
  \checkfont{eurm10}
  \iffontfound
    \IfFileExists{upmath.sty}
      {\typeout{^^JFound AMS Euler Roman fonts on the system,
                   using the 'upmath' package.^^J}%
       \usepackage{upmath}}
      {\typeout{^^JFound AMS Euler Roman fonts on the system, but you
                   don't seem to have the}%
       \typeout{'upmath' package installed. EJM.cls can take advantage
                 of these fonts,^^Jif you use 'upmath' package.^^J}%
       \providecommand\mathup[1]{##1}%
       \providecommand\mathbup[1]{##1}%
      }
  \else
    \providecommand\mathup[1]{#1}%
    \providecommand\mathbup[1]{#1}%
  \fi
\fi

\ifprodtf \else
  \checkfont{msam10}
  \iffontfound
    \IfFileExists{amssymb.sty}
      {\typeout{^^JFound AMS Symbol fonts on the system, using the
                'amssymb' package.^^J}%
       \usepackage{amssymb}%
         \let\leq=\leqslant
         \let\geq=\geqslant
      }{}
  \fi
\fi

\ifprodtf \else
  \IfFileExists{amsbsy.sty}
    {\typeout{^^JFound the 'amsbsy' package on the system, using it.^^J}%
     \usepackage{amsbsy}}
    {\providecommand\boldsymbol[1]{\mbox{\boldmath $##1$}}}
\fi

\newsavebox{\astrutbox}
\sbox{\astrutbox}{\rule[-5pt]{0pt}{20pt}}